\documentclass[10pt,reqno]{amsart}
\usepackage{amsmath,amssymb,bbold,enumerate,graphicx}
\usepackage[arrow,matrix,curve]{xy}

\usepackage{mathrsfs,bbm,eucal}
\theoremstyle{plain}
\newtheorem{thm}{Theorem}[section]
\theoremstyle{plain}
\newtheorem{lem}[thm]{Lemma}
\theoremstyle{plain}
\newtheorem{prop}[thm]{Proposition}
\theoremstyle{plain}

\newtheorem{cor}[thm]{Corollary}
\theoremstyle{definition}

\theoremstyle{remark}

\DeclareMathAlphabet{\mathpzc}{OT1}{pzc}{m}{it}

\newcommand{\binsubseteq}{\mathbin{\text{\rotatebox[origin=c]{270}{$\subseteq$}}‌​}}

\newcommand{\bundle}[1]{\CMcal{#1}}

\newcommand{\C}{\mathbbm{C}}

\newcommand{\Q}{\mathbbm{Q}}

\renewcommand{\H}{\mathscr{H}}
\newcommand{\HK}{\mathcal{H}}
\newcommand{\Ho}{\mathfrak{H}}
\newcommand{\Tr}{\mathpzc{T}}
\newcommand{\Z}{\mathbbm{Z}}

\begin{document}
\title[Novikov conjecture and Hirzebruch L--class]{Remark on the Novikov conjecture and the homotopy invariance of the Hirzebruch L--class}
\author{Mario Listing}

\begin{abstract}
We use K--area homology to summarize some results about the Novikov conjecture and the Hirzebruch $L$--class. In fact, we provide necessary and sufficient conditions for closed manifolds to have a homotopy invariant $L$--class. In order to obtain additional properties we also introduce the cohomology of infinite K--area which preserves the cup product of singular cohomology. 
\end{abstract}
\maketitle

\section{Introduction}
The purpose of this article is to give an alternative approach to the Novikov conjecture. We ignore results about the Baum--Connes assembly map which lead to the strong Novikov conjecture. The definition of the K--area homology in \cite{List10} and most of the results in this paper are motivated by Gromov's work \cite{Gr01}. Remember that the K--area homology is determined by the curvature of complex vector bundles, is easy to compute in many cases and has plenty of interesting properties (cf.~\cite{List10,pre_List12}). We denote by
\[
L(M)=1+L_1(M)+L_2(M)+\cdots \in H^{4*}(M;\Q )
\]
the total Hirzebruch $L$--class of the tangent bundle of $M$. We say that $L(M)$ is a \emph{homotopy invariant of} $M$, if $L(M')=f^*L(M)$ holds for any homotopy equivalence $f:M'\to M$. We use the analogous definition for the class  $L_k(M)\in H^{4k}(M;\Q )$. Note that the homotopy invariance of $L(M)$ is equivalent to the homotopy invariance of all rational Pontryagin classes $p_k(M)\in H^{4k}(M;\Q )$. Novikov proved in \cite{Nov0} the invariance of the $L$--class under homeomorphisms  and presented in \cite{Novikov} counterexamples to its homotopy invariance. Eventually this lead to the conjecture about the oriented homotopy invariance of the higher signatures. The \emph{Novikov conjecture for a discrete group} $\pi $ asserts
\[
h_*(L(M)\cap [M])=(h\circ f)_*(L(M')\cap [M'])\in H_*(B\pi ;\Q )
\]
for all orientation preserving homotopy equivalences $f:M'\to M $ between closed oriented manifolds and continuous maps $h:M\to B\pi $. In this case and below, $B\pi =K(\pi ,1)$ denotes the classifying space of the discrete group $\pi $. Notice that the numbers
\[
\left< \alpha ,h_*(L(M)\cap [M])\right> =\left< h^*\alpha \cup L(M),[M]\right>
\]
are the \emph{higher signatures} of $M$ associated to $h:M\to B\pi $ and $\alpha \in H^*(B\pi ;\Q )$.  

As in \cite{List10, pre_List12} we denote by $\H _*(M;\Q )$ the rational K--area homology of a compact smooth manifold $M$ and extend the functor $\H _*$ in the usual way to the category of topological spaces (cf.~\cite[theorem 6.4]{List10} or section \ref{sec2}).  
\begin{thm}
\label{mainthm}
\begin{enumerate}[(1)]
\item A discrete group $\pi $ with $\H _*(B\pi ;\Q )=0$ satisfies the Novikov conjecture. 
\item Let $M^n$ be an oriented closed manifold with $\H _{n-4k}(M;\Q )=0$ for some $k>0$, then $L_k(M)$  is a homotopy invariant of $M$. 
\item If $M^n $ is an oriented closed manifold with $\H _*(B\pi _1(M);\Q )=0$, then $L(M)$ is a homotopy invariant of $M$ if and only if $\H _{n-4k}(M;\Q )=0$ for all $k>0$.
\end{enumerate}
\end{thm}
Note the if and only if in (3), in fact $\H _{n-4k}(M;\Q )\neq 0$ for some $k>0$ yields a (smooth) homotopy equivalence $f:M'\to M$ with $L_k(M')\neq f^*L_k(M)$. In order to show part (3) we use a result by Davis \cite{Davis} which is well known to experts in case $M$ is simply connected. Notice that simply connected closed manifolds $M$ satisfy $\H _k(M;\Q )=H_k(M;\Q )$ for all $k>0$. We can show  $\H _*(B\pi ;\Q )=0$ for plenty of discrete groups $\pi $, but have no example with $\H _*(B\pi ;\Q )\neq 0$. However, vanishing of $\H _*(B\pi ;\Q )$ is unsettled for many groups $\pi $. Claims (1) and (2) in the theorem follow from a result by Hilsum and Skandalis \cite{HilSka}. Claim (1) can be improved by a straightforward argument of \cite[theorem 3.9]{Hank}, in fact Hanke's theorem implies the strong Novikov conjecture for $\pi $ if $\H _*(B\pi ;\Q )=0$. Hence, $\H _*(B\pi ;\Q )=0$ may be called the \emph{very strong Novikov conjecture} for $\pi $. Observe that plenty of non--aspherical closed manifolds $M$ with $\H _*(M;\Q )=0$ exist. Since $\H _0(M;\Q )$ and $\H _1(M;\Q )$ vanish for all manifolds, theorem \ref{mainthm}(2) includes two classic results: the oriented homotopy invariance of the signature if $\dim M=4k$ and moreover, the homotopy invariance of the class $L_{k}(M)$ if $\dim M=4k+1$ (cf.~\cite{Novikov}). A more subtle argument than theorem \ref{mainthm}(3) yields the following result. 
\begin{cor}
\label{main_cor}
If $M^n$ is oriented and closed with $\H _{*}(B\pi _1(M);\Q )=0$ and $\dim \H _{n-4k}(M;\Q )=1$, $L_k(M)\cap [M]\in \H _{n-4k}(M;\Q )$ for some $k>0$, then there are infinitely many distinct smooth manifolds with the homotopy type of $M$.
\end{cor}
This corollary shows in particular the existence of infinitely many distinct smooth manifolds with the homotopy type of the projective spaces $\mathbbm{Ca}P^2$, $\C P^{m+1}$ and $\mathbbm{H}P^m$ if $m>1$. Examples with infinite fundamental group are $M^n\# T^n$ and $(\mathbbm{R}P^3\# \mathbbm{R}P^3)\times N^4$  if $M$ satisfies the assumptions of the corollary and $|\pi _1(N)|<\infty $.

\section{Cohomology of infinite K--area}
\label{sec2}
The cohomology of infinite K--area is complementary to the homology of finite K--area. The advantage of the cohomology version is the additional ring structure whereas many results are better stated within the homology version. Let $(M,g)$ be a compact Riemannian manifold and $\mathscr{V}_{\epsilon ,g}(M)$ be the set of finite dimensional Hermitian vector bundles $(\bundle{E},\nabla )\to M$ with curvature $\| R^\bundle{E}\| _g<\epsilon $ where $\| .\| _g$ is the $L^\infty $--operator norm on $\Lambda ^2TM\otimes \mathrm{End}(\bundle{E})$. We denote by $H^{2*}_{\epsilon ,g}(M;\Q )\subseteq H^{2*}(M;\Q )$ the vector subspace which is generated by the Chern characters $\mathrm{ch}(\bundle{E})\in H^{2*}(M;\Q )$ of all bundles $(\bundle{E},\nabla )\in \mathscr{V}_{\epsilon ,g}(M)$. The \emph{cohomology of infinite K--area} of $M$ is defined by
\[
\begin{split}
\HK ^{2*}(M;\Q )&=\bigcap _{\epsilon >0}H_{\epsilon ,g}^{2*}(M;\Q )=\lim _{\epsilon \to 0}H^{2*}_{\epsilon ,g}(M;\Q )\\
\HK ^{2*+1}(M;\Q )&=\left\{ \alpha \in H^{2*+1}(M;\Q )\ |\ \alpha \times [S^1]^*\in \HK ^{2*} (M\times S^1;\Q )\right\} 
\end{split}
\]
At first we observe that $\HK ^{*}(M;\Q )$ depends only on the homotopy type of $M$ and not on the choice of the Riemannian metric $g$. In fact, if $f:M\to N$ is continuous, the pull back of (singular) cohomology classes provides a linear map $f^*:\HK ^{k}(N;\Q )\to \HK ^{k}(M;\Q )$ for all $k$. Moreover, if $h:M\to N$ is homotopic to $f$, then $h^*=f^*$. Hence, $\HK ^{*}(\, .\, ;\Q )$ (respectively $\HK ^{j}(\, .\, ;\Q )$) is a cofunctor on the category of compact smooth manifolds and continuous maps into the category of rational vector spaces which satisfies the homotopy axiom. Clearly, $\HK ^*(\{ pt\} ;\Q )=H^0(\{ pt\} ;\Q )=\Q $ and $\HK ^*(M\coprod M';\Q )=\HK ^*(M;\Q )\oplus \HK ^*(M';\Q )$. The following proposition follows immediately from the definitions.
\begin{prop}
The cohomology of infinite K--area is the complement of the homology of finite K--area:
\[
\HK ^k(M;\Q )=\H _k(M;\Q )^{\perp }=\{ \alpha \in H^k(M;\Q )\ |\ \left< \alpha ,\theta \right> =0\ \, \forall \, \theta \in \H _k(M;\Q )\} 
\]
where $\left< .,.\right> $ means the ordinary pairing of singular cohomology and homology (note that $\H _k(M;\Q )$ is a subspace of $H_k(M;\Q )$ by definition).  
\end{prop}
\begin{lem}
\label{lemma22}
$\HK ^{2*}(M;\Q )$ is a ring with respect to the cup product in singular cohomology, in fact
\[
\alpha \cup \beta \in \HK ^*(M;\Q )
\]
for all $\alpha \in \HK ^{2*}(M;\Q )$ and $\beta \in \HK ^*(M;\Q )$. Moreover, the subalgebra generated by $H^1(M;\Q )$ is contained in $\HK ^{*}(M;\Q )$.
\end{lem}
\begin{proof}
If $\bundle{E},\bundle{F}\in \mathscr{V}_{\epsilon ,g}(M)$, then $\bundle{E}\otimes \bundle{F}\in \mathscr{V}_{2\epsilon ,g}(M)$ for the tensor product connection. Hence, $\alpha ,\beta \in H^{2*}_{\epsilon ,g}(M;\Q )$ implies $\alpha \cup \beta \in H^{2*}_{2\epsilon ,g}(M;\Q )$ which proves the first claim if $\alpha ,\beta $ are even classes. If $\beta \in \HK ^{2*+1}(M;\Q )$, then $\beta \times [S^1]\in \HK ^{2*}_{\epsilon ,g\oplus dt^2}(M\times S^1;\Q )$ for all $\epsilon >0$. Thus, $(\alpha \cup \beta )\times [S^1]\in H_{2\epsilon ,g\oplus dt^2}^{2*}(M\times S^1;\Q )$ for all $\epsilon >0$ if $\alpha \in \HK ^{2*}(M;\Q )$. Let $\alpha \in H^{k}(M;\Q )$ be in the subalgebra generated by $H^1(M;\Q )$, then there is a map $f:M\to T^k$ with $\alpha =x\cdot f^*([T^k]^*)$ for some $x\in \Q $. Since $\HK ^*(T^k;\Q )=H ^*(T^k;\Q )$ (cf.~section 3 in \cite{List10}) we conclude the second claim from functoriallity.
\end{proof}
Remember that the intersection product is not preserved in K--area homology. In order to avoid difficulties with the cup product of odd classes we also introduce two stabilized versions of the cohomology of infinite K--area. 
\begin{enumerate}[(i)]
\item Let $[T^j]^*\in H^j(T^j;\Q )$ be a generator and define the subspaces
\[
\HK _j^{k}(M;\Q ):=\{ \theta \in H^k(M;\Q )\ |\ \theta \times [T^j]^*\in \HK ^{k+j}(M\times T^j;\Q )\} 
\]
then $\HK _j^k(M;\Q )\subseteq \HK _l^k(M;\Q )$ holds for all $j\leq l$. This follows from the last lemma and $\HK ^*(T^j;\Q )=H^*(T^j;\Q )$ for all $j$. Hence,
\[
\HK _{rs}^k(M;\Q )=\lim _{j\to \infty }\HK _j^k(M;\Q )\subseteq H^k(M;\Q )
\]
yields a well defined cofunctor. We call $\HK _{rs}^*(M;Q)$ \emph{ring stabilized cohomology of infinite K--area}.
\item Let $j+k$ be even and $\HK _{j,\epsilon ,g}^k(M;\Q )\subseteq H^k(M;\Q )$ be the subspace generated by all $i^*(\mathrm{ch}_{j+k}(\bundle{E})\cap 1\times [T^j])\in H^k(M;\Q )$ where $\bundle{E}\to M\times T^j$ are Hermitian bundles of curvature $\| R^\bundle{E}\| _{g\oplus h}<\epsilon $ and $i:M\to M\times T^j $ is an inclusion map. Note that $\HK _{j,\epsilon ,g}^k(M;\Q )$ does not depend on $i$ or the choice of the Riemannian metric $h$ on $T^j$. The definitions imply $\HK ^k_j(M;\Q )=\lim \limits _{\epsilon \to 0}\HK ^k_{j,\epsilon ,g}(M;\Q )$. Moreover, a standard exercise (cf.~\cite[prop.~5.1]{List10}) yields $\HK ^k_{j,\delta ,g}(M;\Q )\subseteq \HK _{l,\epsilon ,g}^k(M;\Q )$ if $j\leq l$ and $\delta \leq \epsilon $. Hence, we define the \emph{stabilized cohomology of infinite K--area} by
\[
\HK ^k_{st}(M;\Q )=\lim _{\epsilon \to 0 }\lim _{j\to \infty  }\HK ^k_{\epsilon ,g,j}(M;\Q ).
\]
Notice that $\alpha \in \HK _{j,\epsilon ,g}^k(M;\Q )$, $\beta \in \HK _{l,\epsilon ,g}^n(M;\Q )$ satisfy $\alpha \cup \beta \in \HK ^{k+n}_{j+l,2\epsilon ,g}(M;\Q )$, i.e.~$\HK ^*_{st}(M;\Q )$ becomes a ring with respect to the cup product.
\end{enumerate}
\begin{prop}
\label{prop23}
We have
\[
\HK ^k(M;\Q )\subseteq \HK ^{k}_{rs}(M;\Q )\subseteq \HK ^k_{st}(M;\Q )\subseteq H^k(M;\Q )
\]
for all $k$. Moreover, $\HK ^{2*}(M;\Q )$, $\HK ^*_{rs}(M;\Q )$ and $\HK _{st}^*(M;\Q )$ are rings with respect to the cup product of singular cohomology. The stabilized version $\HK ^*_{st}$ is complementary to the stabilized K--area homology introduced in \cite[section 5]{List10}:
\[
\HK _{st}^k(M;\Q )=\H _k^{st}(M;\Q )^\perp 
\]
where $\perp $ refers to the pairing $\left< .,.\right> $ between singular cohomology and homology. 
\end{prop}

Many manifolds satisfy $\HK ^*(M;\Q )=\HK _{st}^*(M;\Q )$, but a general proof for this equality or a counterexample are unknown. In case there are manifolds $M$ with $\HK ^*(M;\Q )\neq \HK _{rs}^*(M;\Q )$, $\HK ^{*}_{rs}(M;\Q )$ has advantages over $\HK ^*(M;\Q )$ concerning the Novikov conjecture, and $\HK ^*_{st}$ (respectively $\H _*^{st}$) provides stronger obstructions to positive scalar curvature (cf.\cite{List10}). However, under surgery the functors $\HK _{st}^*$, $\HK _{rs}^*$ may be less well behaved than $\HK ^*$. In fact,  using a suitable relative version of $\H _*^{st}$  the proof of theorem 1.1 in \cite{pre_List12} breaks down precisely in lemma 4.1. That is why we keep working with the ordinary and the stabilized versions. In order to avoid notational difficulties in the observations below we also introduce the ring stabilized K--area homology. In fact, $\H _k^{rs}(M;\Q )$ consists of homology classes $\theta \in H_k(M;\Q )$ with $\theta \times T^j\in \H _{k+j}(M\times T^j;\Q )$ for all $j$. Of course, $\HK ^*_{rs}(M;\Q )=\H _*^{rs}(M;\Q )^\perp $. 

Since we are interested in statements for classifying spaces $B\pi $ we extend the functors of K--area homology to the category of topological spaces. In fact, if $X$ is a topological space define
\[
\H _k^{\omega  }(X;\Q )=\{ f_*(\theta ) \in H_k(X;\Q )\ |\ f:M\to X, \theta \in \H _k^{\omega }(M;\Q )\} 
\] 
where $f:M\to X$ denotes continuous maps between compact smooth manifolds $M$ and $X$. In this case and below, ''$\omega  $'' my be blank, ring stabilized ''rs'' or stabilized ''st'', i.e.~$\H ^{\omega  }_k(X;\Q )$ stands for the usual K--area homology $\H _k(X;\Q )$ or its stabilized versions $\H _k^{rs}(X;\Q )$, $\H _k^{st}(X;\Q )$. $\H _k^\omega  (X;\Q )$ is a subspace of $H _k(X;\Q )$ as addition is obtained by the disjoint union of compact manifolds. Analogously, the cohomology of infinite K--area of $X$ is given by
\[
\begin{split}
\HK ^k_{\omega  }(X;\Q )&=\H _k^{\omega  }(X;\Q )^\perp \\
&=\{ \alpha \in H^k(X;\Q )\ |\ \forall f:M\to X: f^*(\alpha )\in \HK ^k_\omega (M;\Q )  \}
\end{split}
\]
which is clearly a subspace of $H^k(X;\Q )$. Additionally to the results in \cite{List10,pre_List12} we summarize a few facts:
\begin{enumerate}
\item $\H _*^{\omega  }(\, .\, ;\Q )$ is a functor from the category of topological spaces and continuous maps to the category of rational vector spaces. $\HK ^*_{\omega  }(\, .\, ;\Q )$  is a cofunctor on theses categories.
\item Homotopy invariance: If $f,h:Y\to X$ are homotopic, then $f_*=h_*:\H _k^{\omega  }(Y;\Q )\to \H _k^{\omega  }(X;\Q )$ and $f^*=h^*:\HK ^k_{\omega  }(X;\Q )\to \HK ^k_{\omega  }(Y;\Q )$ hold for all $k$.
\item $\HK ^i_{\omega  }(X;\Q )=H^i(X;\Q )$ and $\H ^{\omega  }_i(X;\Q )=0$ hold for $i=0$ and $i=1$.
\item \label{fact4} Additivity: If $X=\bigvee _\alpha X_\alpha $ is a wedge sum, then the inclusions $i_\alpha :X_\alpha \to X$ provide isomorphisms
\[
 \bigoplus _\alpha \H _k^{\omega  }(X_\alpha ;\Q )\cong \H _k^{\omega  }(X;\Q )
\]
\item \label{fact5} $\HK ^*_{rs}(X;\Q )$, $\HK ^*_{st}(X;\Q )$ and $\HK ^{2*}(X;\Q )$ are rings with respect to the cup product. In fact, the cohomology cross product satisfies
\[
\HK _{\omega }^k(X;\Q )\otimes \HK _{\omega }^l(Y;\Q )\subseteq \HK _{\omega }^{k+l}(X\times Y;\Q )
\]
for $\omega \in \{ rs ,st\} $ and all $k,l$. 
\item \label{fact55} The cap product of singular homology/cohomology satisfies
\[
\HK ^{j}_\omega (X;\Q )\times \H _k^\omega (X;\Q )\stackrel{\cap }\longrightarrow \H _{k-j}^\omega (X;\Q ),\ (\alpha ,\theta )\mapsto \alpha \cap \theta 
\]
for all $j,k$ if $\omega \in \{ rs,st\}$.  $\HK ^j(X;\Q )\times \H _k(X;\Q )\stackrel{\cap }\to \H _{k-j}(X;\Q )$ is well defined if $j(k-j)$ is even. 

\item \label{fact6} If $\tilde X\to X$ is a finite covering, then $f_*:\H _k^{\omega  }(\tilde X;\Q )\to \H _k^{\omega  }(X;\Q )$ is surjective for all $k$. Thus, $\H _*^{\omega  }(\tilde X;\Q )=0$ implies $\HK ^*_{\omega  }(X;\Q )=H^*(X;\Q )$. 
\item \label{fact7}If $N$ is a compact manifold with finite fundamental group and $X$ a topological space, then $\HK _{\omega  }^k(X\times N;\Q )=\HK ^k_{\omega  }(X;\Q )\otimes H^0(N;\Q )$ for all $k$.
\end{enumerate}
\begin{proof}
(1), (2) and (3) are obvious and will be left to the reader.

(\ref{fact4}) The inclusions $i_\alpha :X_\alpha \to X$ and the projections $p_\alpha :X\to X_\alpha $ (here everything outside $X_\alpha $ is mapped to the basepoint $x_\alpha \in X_\alpha $) satisfy $p_\alpha \circ i_\alpha =\mathrm{id}_{X_\alpha }$. $(\bigoplus _\alpha i_\alpha )_*$ is an isomorphism for singular homology if $k>0$, i.e.~functorial reasons provide the claim.  

(\ref{fact5}) Given $\alpha ,\beta \in \HK ^{2*}(X;\Q )$, we have to show that $f^*(\alpha \cup \beta )\in \HK ^{2*}(M;\Q )$ for all $f:M\to X$. Since $f^*(\alpha ),f^*(\beta )\in \HK ^{2*}(M;\Q )$ for all $f:M\to X$, the claim follows from lemma  \ref{lemma22} and $f^*\alpha \cup f^*\beta =f^*(\alpha \cup \beta )$. $\HK ^*_{st}(X;\Q )$ and $\HK ^*_{rs}(X;\Q )$ are rings by the same argument, because $\HK _{st}^*(M;\Q )$ and $\HK ^*_{rs}(M;\Q )$ are rings for all compact $M$. For the second statement we use the projections $p_X:X\times Y\to X$ and $p_Y:X\times Y\to Y$ to show $p_X^*(\alpha )=\alpha \times 1\in \HK ^*_{\omega }(X\times Y;\Q )$ and $p_Y^*(\beta )=1\times \beta \in \HK ^*_{\omega }(X\times Y;\Q )$ for all $\alpha \in \HK ^*_{\omega }(X;\Q )$ and $\beta \in \HK ^*_{\omega }(Y;\Q )$. Hence, the ring structure of $\HK ^*_{\omega }(X\times Y;\Q )$ and
\[
\alpha \times \beta =p_X^*(\alpha )\cup p_Y^*(\beta )\in H^*(X\times Y;\Q )
\] 
provide the claim if $\omega =\{ rs,st\} $.

(\ref{fact55}) Given $\alpha \in \HK ^j_\omega (X;\Q )$ and $\theta \in \H _k^\omega (X;\Q )$, then $\alpha \cup \beta \in \HK ^{k}_\omega (X;\Q )$ for all $\beta \in \HK ^{k-j}_\omega (X;\Q )$ yields
\[
0=\left< \alpha \cup \beta ,\theta \right> =\left< \beta ,\alpha \cap \theta \right> ,
\]
for all $\beta \in \HK ^{k-j}_\omega (X;\Q )$, i.e.~$\alpha \cap \theta \in \H _{k-j}^\omega (X;\Q )$. The same is true for $\HK ^j$ and $\H _k$ if $j$ or $k-j$ is even.

(\ref{fact6}) Note at first that proposition 3.2 in \cite{List10} holds in the nonorientable case using the topological transfer map, in fact $h_*:\H _k^\omega (\tilde M;\Q )\to \H _k^\omega (M;\Q )$ is surjective for finite coverings $h:\tilde M\to M$ of a compact manifold $M$. Given $\theta \in \H _k^\omega (X;\Q )$, there is some $g:M\to X$ and $\eta \in \H _k^\omega (M;\Q )$ with $g_*(\eta )=\theta $. At first we construct a manifold $\tilde M$ which ''covers'' the path components of $M$ such that $g$ lifts to a map $b:\tilde M\to \tilde X$:
\[
\begin{xy}
\xymatrix{\tilde M\ar@{.>}[d]^{h }\ar@{.>}[r]^{b} &\tilde X \ar[d]^f \\
M\ar[r]^{g}&X}
\end{xy}
\]
Since $M$ is compact, $\pi _0(M)$ is represented by a finite set of points $p_1,\ldots ,p_s\in  M$. Choose $\tilde x_i\in f^{-1}(g(p_i))$ and consider the subgroups $H_i:= g_{\#}^{-1}(f_{\#}(\pi _1(\tilde X,\tilde x_i)))$, then $H_i\subseteq \pi _1(M,p_i)$ has finite index for all $i=1,\ldots ,s$. Hence, each path component $M_i$ of $M$ has a finite covering $h_i:\tilde M_i\to M_i$ with $(h_i)_{\#} (\pi _1(\tilde M_i,\tilde p_i))=H_i$ where $h_i(\tilde p_i)=p_i$. Consider the manifold $\tilde M=\coprod _i\tilde M_i$ and the map $h:\tilde M\to M $ induced by the coverings $h_i$, then $h_*:\H _k^\omega (\tilde M;\Q )\to \H _k^\omega (M;\Q )$ is surjective because $(h_i)_*$ is surjective for all $i$. Moreover, $g\circ h$ satisfies  $g_{\#}h_{\#}(\pi _1(\tilde M,\tilde p_i))\subseteq f_{\#}(\pi _1(\tilde X,\tilde x_i))$ for all $i$. Hence, there is a lift $b:\tilde M\to \tilde X$ of the map $g\circ h:\tilde M\to X$. As $h_*:\H _k^\omega (\tilde M;\Q )\to \H _k^\omega (M;\Q )$ is surjective, we can choose $\tilde \eta \in \H _k^\omega (\tilde M;\Q )$ with $h_*(\tilde \eta )=\eta $. Thus,
\[
(f\circ b)_*(\tilde \eta )=(g\circ h)_*(\tilde \eta )=g_*(\eta )=\theta ,
\]
and $\tilde \theta :=b_*(\tilde \eta )$ provide $f_*(\tilde \theta )=\theta $, i.e.~$f_*$ is surjective. 

(\ref{fact7}): Without loss of generality $N$ is connected, the general case follows from the additivity axiom. If $X=M$ is a compact manifold, the claim follows from proposition 3.2(iii) in \cite{pre_List12}. The inclusion $\supseteq $ is obvious. Considering inclusions $i_1:X\to X\times N$ and $i_2:M\to M\times N$ for a point $x\in N$, then $f^*i_1^* =i_2^*(f\times \mathrm{id}_N)^* $ for all $f:M\to X$. If $\theta \in \HK ^{k}_\omega (X\times N;\Q )$, $M$ is compact and $f:M\to X$ is continuous, then $(f\times \mathrm{id}_N)^*\theta \in \HK ^k_\omega (M\times N;\Q )\stackrel{i_2^*}= \HK ^k_\omega (M;\Q )$ which implies $f^*i^*_1\theta \in \HK ^k_\omega (M;\Q )$ for all $f:M\to X$, hence $i_1^*\theta \in \HK ^k_\omega (X;\Q )$. 
\end{proof}
In analogy to the observations in \cite{Ros1} we are able to conclude some basic properties for the classifying spaces $B\pi $ if $\pi $ is a discrete group. In fact, the following proposition shows $\H _*(B\pi ;\Q )=0$ for many discrete groups $\pi $.
\begin{prop}
\begin{enumerate}[(a)]
\item If $\pi _1\subset \pi $ has finite index and $\H _*^{\omega  }(B\pi _1;\Q )=0$, then $\H _*^{\omega  }(B\pi ;\Q )=0$. This follows immediately from fact (\ref{fact6}) using the finite covering $B\pi _1\to B\pi $.
\item Let $\pi _1\subset \pi $ be a finite normal subgroup with $\H _*^{\omega  }(B(\pi /\pi _1);\Q )=0$, then $\H _*^{\omega  }(B\pi ;\Q )=0$. Consider the map $f:B\pi \to B(\pi /\pi _1)$ induced by the homomorphism $\pi \to \pi /\pi _1$, then $f_*:H_k(B\pi ;\Q )\to H_k(B(\pi /\pi _1);\Q )$ is injective, i.e.~functoriallity of $\H _*^\omega  $ completes the proof.
\item \label{refc}Let $H_k(B\pi _2;\Q )$ be finite dimensional for all $k$ and $\H _*^{rs}(B\pi _i;\Q )=0$ for $i=1,2$, then $\H _*^{rs}(B(\pi _1\times \pi _2);\Q )=0$. Here we use the K\"unneth theorem for rational cohomology and fact (\ref{fact5}) above. The same is true for $\H ^{st}_*$.
\item If $\H _*^{\omega  }(B\pi _i;\Q )=0$ for $i=1,2$, then the free product $\pi :=\pi _1*\pi _2$ satisfies $\H _*^{\omega  }(B\pi ;\Q )=0$. This follows from $B(\pi _1*\pi _2)=B\pi _1\vee B\pi _2$ and fact (\ref{fact4}).
\end{enumerate}
\end{prop}
\section{The $L$--class and K--area homology}
\label{sec3}
The $L$--class of a manifold $M$ is a polynomial in the rational Pontryagin classes of its tangent $TM$. Novikov proved in \cite{Nov0} that the rational Pontryagin classes and therefore the Hirzebruch $L$--class $L(M)$ depend only on the homeomorphism type of the manifold. The geometric relevance of $L(M)$ enters as the index of (twisted) signature operators. In fact, $\int _ML(M)$ is an integer and coincides with the signature of a closed oriented manifold. This yields the homotopy invariance of $L_k(M)$ if $\dim M=4k$.  Moreover, Novikov also showed in \cite{Novikov} the homotopy invariance of $L_k(M)\in H^{4k}(M;\Q )$ if $M$ has dimension $4k+1$. However, in general the total $L$--class is not invariant under homotopy equivalences. Below we give necessary and sufficient conditions for manifolds $M$ having a homotopy invariant $L$--class.  

Given an orientation preserving homotopy equivalence $f:M'\to M$ between closed manifolds and a bundle $\bundle{E}\to M$. We consider the twisted signature operators
\[
D^M_{\bundle{E}}:\Gamma (\Lambda ^+TM\otimes \bundle{E})\to \Gamma (\Lambda ^-TM \otimes \bundle{E})
\]
and $D^{M'}_{f^*\bundle{E}}$ on $M$ respectively $M'$.  Hilsum and Skandalis proved in \cite{HilSka} the existence of $\epsilon >0$ such that $\| R^\bundle{E}\| _g<\epsilon $ implies
\[
\left< L(M)\cdot \rho (\mathrm{ch} (\bundle{E})),[M]\right> =\mathrm{ind}(D^M_{\bundle{E}})=\mathrm{ind}(D^{M'}_{f^*\bundle{E}})=\left< L(M')\cdot f^*\rho (\mathrm{ch}(\bundle{E})),[M']\right> .
\]
Here $\rho $ denotes multiplication by $2^k$ on $H^{2k}(M;\Q )$ for all $k$ (cf.~\cite{LaMi} for Atiyah--Singer). Note that the $\epsilon $ depends on the homotopy equivalence $f$ and the choice of a Riemmanian metric $g$ on $M$. Thus, we obtain for all bundles $\bundle{E}\in \mathscr{V}_{\epsilon ,g}(M)$
\[
\left< \rho (\mathrm{ch}(\bundle{E})),L(M)\cap [M]-f_*(L(M')\cap [M'])\right> =0
\]
which yields the following proposition.

\begin{prop}
Let $f:M'\to M$ be an orientation preserving homotopy equivalence of closed oriented manifolds, then the Poincar\'e duals of the $L$--classes satisfy
\[
L(M)\cap [M]-f_*(L(M')\cap [M'])\in \H _{n-4*}^{rs}(M;\Q ).
\]
Alternatively
\[
\left< \alpha \cdot L(M),[M]\right> =\left< f^*\alpha \cdot L(M'),[M']\right> 
\]
holds for all $\alpha \in \HK ^*_{rs}(M;\Q )$.
\end{prop}
\begin{proof}
Choose $i$ such that $i+\dim M$ is even and consider the orientation preserving homotopy equivalence $f\times \mathrm{id}:M'\times T^i\to M\times T^i$. The result by Hilsum and Skandalis yields $\epsilon =\epsilon (g\oplus h,i)>0$ with the property
\[
0=\left< \mathrm{ch}(\bundle{E}),\bigl( \rho _i\bigl( L(M)\cap [M]-f_*(L(M')\cap [M'])\bigl) \bigl) \times [T^i]\right> 
\]
for all $\bundle{E}\to M\times T^i$ with curvature $\| R^\bundle{E}\| _{g\oplus h}<\epsilon $, here $\rho _i$ is multiplication by $2^{\frac{i+k}{2}}$ on $H_k(M;\Q )$. This yields $\bigl( L(M)\cap [M]-f_*(L(M')\cap [M'])\bigl) \times [T^i]\in \HK _{2*}(M\times T^i;\Q )$ for all $i$ which completes the proof. Note the problem trying to prove this proposition for $\H ^{st}_*$: the choice of $\epsilon >0$ depends on $i$.
\end{proof}
Hence, assuming $\H _{n-4k}^{rs}(M;\Q )=0$ for $k>0$ yields $L_k(M')=f^*L_k(M)$ for all homotopy equivalences $f:M'\to M$. Together with proposition \ref{prop23} this proves theorem \ref{mainthm}(2). Of course, $\H _{n-4k}^{rs}(M;\Q )=0$ for all $k>0$ implies the homotopy invariance of $L(M)$. Let $\pi $ be a discrete group with $\H _*^{rs}(B\pi ;\Q )=0$, $h:M\to B\pi $ be continuous and $f:M'\to M$ be an orientation preserving homotopy equivalence, then the last proposition provides the Novikov conjecture for $\pi $ and therefore proves theorem \ref{mainthm}(1):
\[
\left< h^*\alpha \cdot L(M),[M]\right> =\left< (h\circ f)^*\alpha \cdot L(M'),[M']\right>  .
\]
for all $\alpha\in \HK ^*_{rs}(B\pi ;\Q )=H^*(B\pi ;\Q )$.

Let $\Gamma _*M$ be the (discrete) set of vectors
\[
f_*(L(M')\cap [M'])-h_*(L(M'')\cap [M''])\in H_*(M;\Q )
\]
where $f:M'\to M$ and $h:M''\to M$ run over all orientation preserving smooth homotopy equivalences to $M$, then the last proposition shows
\[
\Gamma _*M\subset \H _*^{rs}(M;\Q )\subseteq \H _*(M;\Q ). 
\]
Observe that $\Gamma _{*}M$ depends only on the homotopy type of $M$, i.e.~if $g:M\to N$ is a homotopy equivalence between orientable closed manifolds, then $g_*:\Gamma _*M\to \Gamma _*N$ is a bijection. The following theorem is known to experts in various forms, we use a version by Davis. Notice that the assumption $n>4$ is superfluous in our present situation, and the proof needs only the functoriallity of $\H ^\omega _*$ and \cite[theorem 6.5]{Davis}.
\begin{thm}[\cite{Davis}]
Let $M^n$ be oriented and closed with $\H _*^\omega (B\pi _1(M);\Q )=0$, then for any $\theta \in \bigoplus\limits _{k>0}\H _{n-4k}^\omega (M;\Q )$ there is an integer $r\neq 0$ such that $r\Z \cdot \theta \subset \Gamma _*M$. 
\end{thm}

Observe that the condition $k>0$ is essential because $\Gamma _nM=0$ for all $M^n$ whereas $\H _n(M;\Q )$ is nontrivial in many cases. This theorem completes the proof of theorem \ref{mainthm}. In fact, if $\H _{n-4k}(M;\Q )\neq 0$ for some $k>0$, then $\Gamma _{n-4k}M$ is nontrivial and the last theorem yields an infinite number of (distinct) smooth homotopy equivalences $f:M'\to M$ with $f^*L_k(M)\neq L_k(M')\in H^{4k}(M';\Q )$. 
\begin{cor}
Let $M^n$ be oriented and closed, then $\H _*^{rs}(B\pi _1(M);\Q )=0$ implies
\[
\bigoplus _{k>0}\H ^{rs}_{n-4k}(M;\Q )=\Q \cdot \Gamma _*M.
\]
This is also true for the functor $\H _*$. 
\end{cor}
It remains to prove corollary \ref{main_cor}. Let $g:M\to M$ be an orientation preserving homotopy equivalence, then $g_*=\pm \mathrm{Id} $ on $\H _{n-4k}(M;\Q )$ if $\dim \H _{n-4k}(M;\Q )=1$ (consider the integral version on $\H _{n-4k}(M)$ and $\H _{n-4k}(M)\otimes \Q =\H _{n-4k}(M;\Q )$, here we use $\mathrm{Gl}_1(\Z )=\Z _2$). Hence, under the assumptions of corollary \ref{main_cor}, $g_*-\tilde g_*$ vanishes or is multiplication by $\pm 2$ on $\H _{n-4k}(M;\Q )$ for homotopy equivalences $g,\tilde g:M\to M$. Thus, if $L_{k}(M)\cap [M]\in \H _{n-4k}(M;\Q )$, the group of orientation preserving homotopy equivalences $M\to M$ determines at most three points in $\Gamma _{n-4k}M$. Since the conditions in the corollary are homotopy invariant, this remains true for any manifold homotopy equivalent to $M$. Because $\Gamma _{n-4k}M $ has an infinite number of elements, there are infinitely many distinct smooth manifolds with the homotopy type of $M$.

\begin{thm}
\label{thm35}
\begin{enumerate}[(a)]
\item Theorem \ref{mainthm} and corollary \ref{main_cor} remain valid for the functor $\H _*^{rs}$ instead of $\H _*$.
\item Let $M$ be oriented closed with $\H _*^{(rs)}(B\pi _1(M);\Q )=0$. If $j\geq 2$, then
\[
P:\H _k^{(rs)}(M;\Q )\oplus H_{k-j}(M;\Q )\to \Q \cdot \Gamma _k(M^n\times S^j)\ ,\ (\theta ,\eta )\mapsto \theta +\eta \times [S^j]
\]
is a well defined isomorphism for all $k<n+j$ with $n+j-k\in 4\Z $. 
\item If $\H _*(B\pi ;\Q )$ is trivial, then $\H _{*}^{rs}(M;\Q )=\H _{*}(M;\Q )$ holds for all closed manifolds $M$ with fundamental group $\pi $.
\end{enumerate}
\end{thm}
\begin{proof}
a) is obvious by the above considerations.

b) If $j\geq 2$, $P$ is injective with image $\H _k^{(rs)}(M\times S^j;\Q )$ (cf.~\cite[Remark 6.3]{List10}). Hence, the last corollary yields the claim because $k=n+j-4s$ for $s>0$.  

c) Using fact (\ref{fact6}) from section \ref{sec2} it suffices to consider the oriented case. Since $\H _*(B\pi ;\Q )$ and $\H ^{rs}_*(B\pi ;\Q )=0$ vanish, $\H _{n-4k}^{rs}(M;\Q )=\H _{n-4k}(M;\Q )$ for all $k>0$ follows immediately from the last corollary. The remaining cases are obtained by b): Consider the image of $H_{k-j}(M;\Q )$ under $P$ and take the quotient yields
\[
\H ^{rs}_k(M;\Q )\cong \Q \cdot \Gamma _k(M\times S^j) /P(H_{k-j}(M;\Q ))\cong \H _k(M;\Q )
\]
if $j\geq 2$ and $k<n+j$ with $n+j-k\in 4\Z $. Hence, using $j=2,3,4,5$ provides $\H ^{rs}_k(M;\Q )=\H _k(M;\Q )$ for all $k$.
\end{proof}
Part b) of this theorem provides an alternative definition of K--area homology for closed oriented smooth manifolds if $\H _*(B\pi ;\Q )$ vanishes. However, in general it should be significant easier to compute $\H _*(M;\Q )$ by the curvature of vector bundles instead computing $\Gamma _*(M\times S^j)$. 
\section{Outline to Hanke's K--area for Hilbert module bundles}
We want to sketch a generalization of the above results to the HK--area introduced by Hanke in \cite{Hank}, here HK--area  may stand for Hanke's K--area or K--area for Hilbert modules. We leave the details to experts in C$^*$--algebras and KK--theory. As the Chern character between rational K--theory and rational even (co)homology is an isomorphism for compact manifolds, it suffices to work with (co)homology classes. Let $A$ be a C$^*$--algebra and $\bundle{E}\to M$ be a finitely generated Hilbert A--module bundle, then $(\Gamma (\bundle{E}),\phi ,0)$ defines an element in $[\bundle{E}]\in KK(\C ,C(M)\otimes A)$, here $\phi :\C \to \mathcal{B}(\Gamma (\bundle{E}))$ means the standard embedding. Consider the Kasparov product
\[
KK(\C ,C(M)\otimes A)\times K^0(A)\stackrel{\bullet }\longrightarrow K^0(M)\stackrel{ch}\longrightarrow H^{2*}(M;\Q ),
\]
where $KK(A,\C )=K^0(A)$ and $KK(C(M),\C )=K^0(M)$ are identified. Given $\epsilon >0$, we denote by $\HK _{H,\epsilon ,g}^{2*}(M;\Q )\subseteq H^{2*}(M;\Q )$ the linear subspace which is generated by all $\mathrm{ch}([\bundle{E}]\bullet \psi )$ considering all C$^*$-algebras $A$, $\psi \in K^0(A)$ and   finitely generated Hilbert $A$--module bundles $\bundle{E}$ carrying a holonomy representation which is $\epsilon $--close to the identity at scale $\ell $. The \emph{cohomology of infinite HK--area} is defined by
\[
\HK ^{2*}_H(M;\Q )=\lim _{\epsilon \to 0}\HK ^{2*}_{H,\epsilon ,g}(M;\Q ). 
\]
If $\H ^H_{2*}(M;\Q )\subseteq H_{2*}(M;\Q )$ denotes the subspace of rational homology classes with finite HK--area in the sense of \cite[Definition 3.5]{Hank}, one obtains again $\HK ^{2*}_H(M;\Q )=\H ^H_{2*}(M;\Q )^\perp $.  Let $\bundle{E}_i$ be finitely generated Hilbert $A_i$--module bundles for $i=1,2$, and assume that $\bundle{E}_i$ has a holonomy representation $\Ho _i:\mathcal{P}_1(M)\to \Tr (\bundle{E}_i)$ which is $\epsilon $--close to the identity at scale $\ell $, then $\bundle{E}_1\otimes \bundle{E}_2$ is a finitely generated Hilbert $A_1\otimes A_2$--module and $\Ho _1\otimes \Ho _2:\mathcal{P}_1(M)\to \Tr (\bundle{E}_1\otimes \bundle{E}_2)$ determines a holonomy representation which is $2\epsilon $--close to the identity at scale $\ell $. This proves $\alpha \cup \beta \in \HK ^{2*}_{H,2\epsilon ,g}(M;\Q )$ for all $\alpha ,\beta \in \HK ^{2*}_{H,\epsilon ,g}(M;\Q )$ and therefore, $\HK ^{2*}_H(M;\Q )$ is a ring with respect to the cup product. We define the \emph{ring stabilized} versions by
\[
\H ^{Hrs}_k(M;\Q )=\{ \theta \in H_k(M;\Q )\ |\ \theta \times [T^i]\in \H _{k+i}^H(M\times T^i;\Q ) \ \forall i+k\in 2\Z \}
\]
and by $\HK _{Hrs}^k(M;\Q )=\H _k^{Hrs}(M;\Q )^\perp $, then $\HK _{Hrs}^*(M;\Q )$ is a ring with respect to the cup product. We extend $\HK _{H}^*$, $\HK _{Hrs}^*$, $\H ^H_*$ and $\H _*^{Hrs}$ as in section \ref{sec2} to the category of topological spaces. This provides the cup and cap product properties of facts (\ref{fact5}) and (\ref{fact55}) in section \ref{sec2}. Whether the covering result of fact (\ref{fact6}) can be extended to $\H ^H_*$ is left to the reader.
\begin{prop}[\cite{HaSch3}]
Let $M$ be closed, $p:\tilde M\to M$ be the universal covering and $V$ be the kernel of $p^*:H^2(M;\Q )\to H^2(\tilde M;\Q )$. Then the subalgebra generated by $V$ is contained in $\HK _{H}^{2*}(M;\Q )$. Hence, $\H _2^H(M;\Q )$ is determined by the image of the Hurewicz map $\pi _2(M)\otimes \Q \to H _2(M;\Q )$. 
\end{prop}
If $M$ has residually finite fundamental group, this remains true for $\HK ^*(M;\Q )$, in fact $\H _2(M;\Q )$ is given by the image of $\pi _2(M)\otimes \Q \to \H _2(M;\Q )$. Because the result of Hilsum and Skandalis in \cite{HilSka} is proved for Hilbert $A$--module bundles, the observations in section \ref{sec3} and therefore also the main theorem extend to the functors $\H ^H_*$ and $\H _*^{Hrs}$. The definitions show
\[
\begin{array}{ccc}
\HK ^k(M;\Q )&\subseteq &\HK ^k_H(M;\Q ) \\
\binsubseteq &&\binsubseteq \\
\HK ^k_{rs}(M;\Q )&\subseteq &\HK ^k_{Hrs}(M;\Q )
\end{array}
\]
for all $k$. Because theorem \ref{thm35} generalizes to $\H ^{H}_*$ and $\H ^{Hrs}_*$, we obtain equality in this diagram, i.e.~$\HK ^k(M;\Q )=\HK _{Hrs}^k(M;\Q )$ for all $k$, if the fundamental group of $M$ satisfies $\H _*(B\pi _1(M);\Q )=0$.
\bibliographystyle{abbrv}
\bibliography{/home/benutzer/math/bib/bibliothek}

\def\cprime{$'$}
\begin{thebibliography}{10}

\bibitem{Davis}
J.~F. {Davis}.
\newblock {Manifold aspects of the Novikov conjecture.}
\newblock In {\em {Surveys on surgery theory. Vol. 1: Papers dedicated to C. T.
  C. Wall on the occasion to his 60th birthday}}, pages 195--224. Princeton,
  NJ: Princeton University Press, 2000.

\bibitem{Gr01}
M.~Gromov.
\newblock Positive curvature, macroscopic dimension, spectral gaps and higher
  signatures.
\newblock In {\em Functional analysis on the eve of the 21st century, Vol.\ II
  (New Brunswick, NJ, 1993)}, volume 132 of {\em Progr. Math.}, pages 1--213.
  Birkh\"auser Boston, Boston, MA, 1996.

\bibitem{Hank}
B.~Hanke.
\newblock {Positive scalar curvature, {K}-area and essentialness.}
\newblock In {\em Global differential geometry}, volume~17 of {\em Springer
  Proceedings in Mathematics}, pages 275--302. {Berlin: Springer}, 2012.

\bibitem{HaSch3}
B.~{Hanke} and T.~{Schick}.
\newblock {The strong Novikov conjecture for low degree cohomology.}
\newblock {\em {Geom. Dedicata}}, 135:119--127, 2008.

\bibitem{HilSka}
M.~{Hilsum} and G.~{Skandalis}.
\newblock {Invariance par homotopie de la signature \`a coefficients dans un
  fibr\'e presque plat. (Homotopy invariance of signature with coefficients in
  a flat bundle).}
\newblock {\em {J. Reine Angew. Math.}}, 423:73--99, 1991.

\bibitem{LaMi}
H.~B. Lawson, Jr. and M.-L. Michelsohn.
\newblock {\em Spin geometry}, volume~38 of {\em Princeton Mathematical
  Series}.
\newblock Princeton University Press, Princeton, NJ, 1989.

\bibitem{List10}
M.~Listing.
\newblock Homology of finite {K}--area.
\newblock {\em Math. Zeit.}, 275(1):91--107, 2013.

\bibitem{pre_List12}
M.~Listing.
\newblock Relative k--area homology and applications.
\newblock {\em arXiv:1310.0755}, 2013.

\bibitem{Nov0}
S.~{Novikov}.
\newblock {Topological invariance of rational Pontrjagin classes.}
\newblock {\em {Sov. Math., Dokl.}}, 6:921--923, 1965.

\bibitem{Novikov}
S.~{Novikov}.
\newblock {Rational Pontryagin-classes. Homeomorphism and homotopy type of
  closed manifolds. I.}
\newblock {\em {Transl., Ser. 2, Am. Math. Soc.}}, 66:214--230, 1968.

\bibitem{Ros1}
J.~Rosenberg.
\newblock {$C^{\ast} $}-algebras, positive scalar curvature, and the {N}ovikov
  conjecture.
\newblock {\em Inst. Hautes \'Etudes Sci. Publ. Math.}, 58:197--212, 1983.

\end{thebibliography}

\end{document}